\documentclass{article}

\usepackage[latin1]{inputenc}
\usepackage{t1enc}
\usepackage{amsmath, amssymb, amsfonts, amsthm, amsopn,epsfig}
\usepackage[none]{hyphenat}
\usepackage{tikz}
\usepackage{float}
\usepackage{graphicx}
\usepackage{caption}
\usepackage[margin=1in]{geometry}
\usepackage[toc,page]{appendix}
\usepackage{epstopdf}
\usepackage{etoolbox}
\usepackage[colorlinks=true]{hyperref}
\usepackage{csquotes}

\theoremstyle{plain}
\newtheorem{theorem}{Theorem}

\newtheorem{corollary}[theorem]{Corollary}
\newtheorem{claim}[theorem]{Claim}
\newtheorem{lemma}[theorem]{Lemma}
\newtheorem{conjecture}[theorem]{Conjecture}

\theoremstyle{definition}

\author{Istv\'{a}n Tomon\thanks{\'{E}cole Polytechnique F\'{e}d\'{e}rale de Lausanne, Research partially supported by Swiss National Science Foundation grants no. 200020-162884 and 200021-175977.			
		\emph{e-mail}: \textbf{istvan.tomon@epfl.ch}}
}
\title{Forbidden induced subposets in the grid}

\begin{document}
\sloppy
\maketitle

\begin{abstract}
In this short paper, we prove the following generalization of a result of Methuku and P\'{a}lv\"{o}lgyi. Let $P$ be a poset, then there exists a constant $C_{P}$ with the following property. Let $k$ and $n$ be arbitrary positive integers such that $n$ is at least the dimension of $P$, and let $w$ be the size of the largest antichain of the grid $[k]^{n}$ endowed with the usual pointwise ordering. If $S$ is a subset of $[k]^{n}$ not containing an induced copy of $P$, then $|S|\leq C_{P}w$.
\end{abstract}

\section{Introduction}

The \emph{Boolean lattice} $2^{[n]}$ is the power set of $[n]=\{1,...,n\}$ ordered by inclusion. If $P$ and $Q$ are posets, a subset $P'$ of $Q$ is a \emph{copy} of $P$ if the subposet of $Q$ induced on $P'$ is isomorphic to $P$.  The following theorem, originally conjectured by Katona, and Lu and Milans \cite{LM} was proved by Methuku and P\'{a}lv\"{o}lgyi \cite{MP}.

\begin{theorem}(Methuku, P\'{a}lv\"{o}lgyi \cite{MP})\label{mainthm0}
	Let $P$ be a poset.  Then there exists a constant $C=C(P)$ with the following property. If  $S\subset 2^{[n]}$ such that $S$ does not contain a copy of $P$, then $|S|\leq C\binom{n}{\lfloor n/2\rfloor}$.
\end{theorem}

The aim of this paper is to give a slightly different proof of this theorem, which extends from the Boolean lattice to arbitrary \emph{grids} as well. A \emph{grid} of size $k_{1}\times...\times k_{n}$ is the cartesian product  $G=[k_{1}]\times...\times[k_{n}]$ endowed with the pointwise ordering, that is, if $(x_{1},...x_{n}),(y_{1},...,y_{n})\in G$, then $(x_{1},...,x_{n})\leq_{G} (y_{1},...,y_{n})$ if $x_{i}\leq y_{i}$ for $i\in [n]$. If $k_{1}=...=k_{n}=k$, we shall write $[k]^{n}$ instead of $[k_{1}]\times...\times [k_{n}]$.

Before we state our main theorem, let us introduce a few definitions. The \emph{dimension} (Dushnik-Miller dimension) of a poset $P$, denoted by $\dim P$, is the smallest positive integer $d$ such that $P$ is the intersection of $d$ linear orders. Formally, $\dim P$ is the smallest positive integer $d$ for which there exist $d$ bijections $L_{1},...,L_{d}:P\rightarrow [|P|]$ such that for every $p,q\in P$, $p\leq_{P}q$ iff $L_{i}(p)\leq L_{i}(q)$ holds for every $i\in [d]$. 

Also, the \emph{width} of a poset $P$ is the size of the largest antichain in $P$ and is denoted by $w(P)$. Let us mention that by a result of Hiraguchi \cite{H}, we have $\dim P\leq\min\{|P|/2,w(P)\}$. 

The following theorem is the main result of this manuscript.

\begin{theorem}\label{mainthm}
	Let $P$ be a poset.  Then there exists a constant $C_{P}$ with the following property. Let $k$ and $n$ be positive integers satisfying $n\geq \dim P$, and let $w$ be the width of $[k]^{n}$. If $S\subset [k]^{n}$ such that $S$ does not contain a copy of $P$, then $|S|\leq C_{P}w$.
\end{theorem}

In case $k=2$, we have $w=\binom{n}{\lfloor n/2\rfloor}$ by the well known theorem of Sperner \cite{S}, so our main theorem is indeed a strengthening of Theorem \ref{mainthm0}.

Forbidden (weak) subposet and induced subposet problems are extensively studied in the Boolean lattice $2^{[n]}$, for recent developments see \cite{BJ,BN,GMT,LM}, for example. However, there are not many such results when $2^{[n]}$ is replaced with some grid $[k]^{n}$. A well known result of Erd\H{o}s \cite{E} is that if $S\subset 2^{[n]}$ does not contain a chain of size $l$, then $|S|$ is at most the sum of the $l-1$ largest binomial coefficients of order $n$. In fact, $|S|\leq (l-1)\binom{n}{\lfloor n/2\rfloor}$. This result easily generalizes to $[k]^{n}$ as well, that is, if $S\subset [k]^{n}$ does no contain a chain of size $l$, then $|S|\leq (l-1)w([k]^{n})$. See Section \ref{sect:lubell} for a short explanation.

The proof of Theorem \ref{mainthm} is quite short, although it builds on two other results, both of which requires its own introduction. The first of these results, presented in Section \ref{sect:patterns}, is a high dimensional generalization of the celebrated theorem of Marcus and Tardos \cite{MT}  on matrix patterns. Let us note that the connection between high dimensional permutation matrices and partial orders has been already explored in \cite{MP} in the proof of Theorem \ref{mainthm0}. The new ingredient in the proof of Theorem \ref{mainthm} is the application of special chain decompositions of the grid. This topic shall be introduced in Section \ref{sect:chains}. The proof of the main theorem is presented in Section \ref{sect:thm}. Finally, in Section \ref{sect:lubell}, we discuss a possible extension of Theorem \ref{mainthm} concerning the so called Lubell function.

\section{Forbidden matrix patterns}\label{sect:patterns}

In this section, we state the theorem of Klazar and Marcus \cite{KM} on forbidden matrix patterns. For this, let us introduce a few definitions.  

A $d$-dimensional $0-1$ matrix is called \emph{$d$-pattern}. The \emph{weight} of a $d$-pattern $M$ is the number of $1$'s in $M$ and is denoted by $\omega(M)$. A $d$-pattern $M$ of size $m_{1}\times...\times m_{d}$ \emph{contains} a $d$-pattern $A$ of size $a_{1}\times...\times a_{d}$, if  there exist indices $i_{x,y}$ for $x\in [d]$, $y\in [a_{x}]$ such that $1\leq i_{x,1}<i_{x,2}<...<i_{x,a_{x}}\leq m_{a}$, and $M(i_{1,y_{1}},...,i_{d,y_{d}})=1$ if $A(y_{1},...,y_{d})=1$. In other words, $M$ contains $A$ if $M$ has an $a_{1}\times...\times a_{d}$ sized submatrix $M'$, where each $1$ of $A$ corresponds to a $1$ of $M'$. Say that $M$ \emph{avoids} $A$, if $M$ does not contain $A$.

 A $k\times...\times k$ sized $d$-pattern $A$ is a \emph{permutation pattern} if each axis-parallel hyperplane of $A$ contains at most one $1$ entry.

 The following is the main theorem of this section, proved by Marcus and Tardos \cite{MT} in the case $d=2$, and extended by Klazar and Marcus \cite{KM} for arbitrary $d$.
 
\begin{theorem}\label{pattern}(Klazar, Marcus \cite{KM})
	Let $A$ be a $d$-dimensional permutation pattern. There exists a constant $c_{A}$ such that for every positive integer $m$, every  $m\times...\times m$ sized $d$-pattern $M$ that avoids $A$ satisfies $$\omega(M)\leq c_{A}m^{d-1}.$$
\end{theorem}

We shall use the following simple corollary of this theorem.

\begin{corollary}\label{cor:pattern}
	Let $P$ be a $d$-dimensional poset. There exists a constant $c_{P}$ such that for every positive integer $m$, if a set $S\subset [m]^{d}$ does not contain a copy of $P$, then $$|S|\leq c_{P}m^{d-1}.$$
\end{corollary}

\begin{proof}
Let $a=|P|$ and let $L_{1},...,L_{d}:P\rightarrow [a]$ be $d$ linear orders, whose intersection is $P$. Define the $a\times...\times a$ sized $d$-pattern $A$ such that 
$$A(i_{1},...,i_{d})= \begin{cases} 1 &\mbox{if } (i_{1},...,i_{d})=(L_{1}(p),...,L_{d}(p))\mbox{ for some }p\in P, \\
0 &\mbox{otherwise.} \end{cases}$$

Clearly, $A$ is a permutation pattern. We show that $c_{P}=c_{A}$ suffices, where $c_{A}$ is the constant defined in Theorem \ref{pattern}. Let $M$ be the $m\times...\times m$ sized $d$-pattern, where $M(j_{1},...,j_{d})=1$ iff $(j_{1},...,j_{d})\in S$. It is easy to see that $\omega(M)=|S|$, and as $S$ does not contain a copy of $P$, $M$ avoids $A$. Hence, $|S|\leq c_{P}m^{d-1}$.
\end{proof}

The optimal order of the constant $c_{A}$ in Theorem \ref{pattern} is also studied. The best general bounds are due to Geneson and Tian \cite{GT}, who proved that $c_{A}$ can be chosen to be at most $2^{O_{d}(k)}$, where $k=\omega(A)$. Also, for each $k,d$ they showed the existence of a $k\times...\times k$ sized $d$-dimensional  permutation pattern $B$ such that $c_{B}$ needs to be at least $2^{\Omega_{d}(k^{1/d})}$, extending a result of Fox \cite{F}. A similar construction shows that $c_{P}$ is also at least $2^{\Omega(|P|^{1/2})}$ for certain $2$-dimensional posets $P$. This already shows that if the constant $C_{P}$ in Theorem \ref{mainthm} truly exists, then it has to be exponential in $|P|$ for certain posets $P$. On the other hand, it is not clear whether there exist posets $P$ such that the order of the optimal constant $C(P)$ in Theorem \ref{mainthm0} is also exponential in $|P|$.

\section{Decomposition into long chains}\label{sect:chains}

The following conjecture was proposed by F\"{u}redi \cite{F}.

\begin{conjecture}\label{furediconj}
	For every positive integer $n$, the Boolean lattice $2^{[n]}$ can be partitioned into $\binom{n}{\lfloor n/2\rfloor}$ chains such that the size of each chain is $l$ or $l+1$, where $l=\lfloor2^{n}/\binom{n}{\lfloor n/2\rfloor}\rfloor$.
\end{conjecture}

 If this conjecture is true, then there exists a partition of $2^{[n]}$ into chains of size $(\sqrt{\pi/2}+o(1))\sqrt{n}$. While Conjecture \ref{furediconj} is still open, the author of this paper \cite{T} proved that $2^{[n]}$ can be partitioned into $\binom{n}{\lfloor n/2\rfloor}$ chains such that the size of each chain in the partition is between $0.8\sqrt{n}$ and $25\sqrt{n}$. This result is deduced from a more general one, which we shall state after introducing some further terminology.
 
 In what comes, we shall define the notion of \emph{unimodal normalized matching poset}, which itself requires a few preliminary definitions. A poset $Q$ is \emph{graded} if there exists a partition of its elements into subsets $A_{0},A_{1},\ldots,A_{n}$ such that $A_{0}$ is the set of minimal elements, and whenever $x\in A_{i}$ and $y\in Q$ such that $x<y$ with no $u\in Q$ satisfying $x<u<y$, then $y\in A_{i+1}$. If there exists such a partition, then it is unique and $A_{0},A_{1},\ldots,A_{n}$ are the \emph{levels} of $Q$.
 
 A graded poset $Q$ with levels $A_{0},...,A_{n}$ is \emph{unimodal} if $|A_{0}|,...,|A_{n}|$ is a unimodal sequence, that is, there exists $m\in \{0,...,n\}$ such that $|A_{0}|\leq\ldots\leq |A_{m}|$ and $|A_{m}|\geq |A_{m+1}|\geq\ldots\geq |A_{n}|$. Also, $Q$ is \emph{rank-symmetric}, if $|A_{i}|=|A_{n-i}|$ for $i=0,\ldots,n$.
 
 A graded poset $Q$ with levels $A_{0},...,A_{n}$ is a \emph{normalized matching poset}, if for every $0\leq i,j\leq n$ and $X\subset A_{i}$, we have
 $$\frac{|X|}{|A_{i}|}\leq \frac{|\Gamma(X)|}{|A_{j}|},$$
 where $\Gamma(X)$ is the set of elements in $A_{j}$ which are comparable with an element of $X$. Also, $Q$ has the \emph{LYM-property}, if every antichain $S\subset Q$ satisfies $\sum_{i=0}^{n}|S\cap A_{i}|/|A_{i}|\leq 1$. By a result of Kleitman \cite{K}, $Q$ is a normalized matching poset if and only if it has the LYM property.

It is easy to show that the Boolean lattice $2^{[n]}$ is a rank-symmetric, unimodal normalized matching poset. The following extension of Conjecture \ref{furediconj} was proposed by Hsu, Logan and Shahriari \cite{HLS}.

\begin{conjecture}\label{genconj}
	Let $Q$ be a rank-symmetric, unimodal normalized matching poset of width $w$. Then $Q$ can be partitioned into $w$ chains, each chain in the partition having size $l$ or $l+1$, where $l=\left\lfloor|P|/w\right\rfloor$.
\end{conjecture} 

Naturally, this conjecture is open as well. The author of this paper \cite{T} proved the following result concerning Conjecture \ref{genconj}.

\begin{theorem}(Tomon \cite{T})\label{unmp}
		Let $Q$ be a unimodal normalized matching poset of width $w$. Then there exists a chain partition of $Q$ into $w$ chains such that the size of each chain in the partition is at least $\frac{|Q|}{2w}-\frac{1}{2}$. Also, there exists a chain partition of $Q$ into $w$ chains such that the size of each chain in the partition is at most $\frac{2|Q|}{w}+5$.
\end{theorem}

Let us remark that the poset $Q$ in Theorem \ref{unmp} need not be rank-symmetric. We shall use the following simple corollary of this theorem.

\begin{corollary}\label{chain}
	Let $k,n$ be positive integers and let $w$ be the width of the grid $[k]^{n}$. Then $[k]^{n}$ can be partitioned into chains such that the size of each chain $C$ in the partition satisfies $$\frac{k^{n}}{4w}\leq |C|\leq\frac{k^{n}}{w}.$$ 	
\end{corollary} 

\begin{proof}
	Note that $[k]^{n}$ is graded with levels $$A_{i}=\{(a_{1},...,a_{n})\in [k]^{n}:a_{1}+...+a_{n}=n+i\}$$ for $i=0,...,kn-n$.  It is well known that $[k]^{n}$ is a rank-symmetric, unimodal normalized matching poset, see p.60-63 in the book of Anderson \cite{A}, for example. Hence, by Theorem \ref{unmp}, $[k]^{n}$ can be partitioned into chains such that the size of each chain is at least $\max\{k^{n}/2w-1/2,1\}\geq  k^{n}/4w$. If the size of a chain $C$ in this partition is larger than $k^{n}/w$, then cut $C$ into smaller pieces such that the size of each piece is at least $k^{n}/4w$, and at most $k^{n}/w$. The resulting chain partition suffices.
\end{proof}

\section{Proof of the main theorem}\label{sect:thm}

For the proof of Theorem \ref{mainthm}, we need the following estimate on the width of $[k]^{n}$.

\begin{lemma}\label{estimate}
Let $k,n$ be positive integers such that $k\geq 2$. Then $w([k]^{n})=\Theta(k^{n-1}/\sqrt{n}).$	
\end{lemma}

\begin{proof}
	We shall use the following bound on the width of grids, which can be found on p.63-68 in \cite{A}. Let $k_{1},...,k_{n}\geq 2$ be integers and let $A=\sum_{i=1}^{n}(k_{i}^{2}-1)$. Then the width of the grid $[k_{1}]\times...\times[k_{n}]$ is $\Theta(k_{1}...k_{n}/\sqrt{A})$. Setting $k_{1}=...=k_{n}=k$, we get the desired result. 
\end{proof}

In the proof of our main theorem, we shall exploit that $[k]^{n}$ is the \emph{cartesian product of posets}. For $i\in [n]$, let $P_{i}=(X_{i},<_{i})$ be posets, then $P_{1}\times...\times P_{n}=(X_{1}\times...\times X_{n},\leq)$ is the cartesian product of $P_{1},...,P_{n}$, where $(p_{1},...,p_{n})\leq (q_{1},...,q_{n})$, if $p_{i}\leq_{i}q_{i}$ for $i\in [n]$. It is easy to see that if $n=n_{1}+...+n_{d}$, then the grid $[k]^{n}$ is isomorphic to the cartesian product $[k]^{n_{1}}\times...\times [k]^{n_{d}}$. With a slight abuse of notation, we shall identify the cartesian product with $[k]^{n}$ itself.

Now we are ready to prove Theorem \ref{mainthm}.

\begin{proof}[Proof of Theorem \ref{mainthm}.] If $k=1$, then $[k]^{n}$ has only one element, so suppose that $k\geq 2$. Let $d=\dim P$ and let $c_{P}$ be the constant defined in Corollary \ref{cor:pattern}. Write $n=n_{1}+...+n_{d}$, where $n_{i}\in\{\lfloor n/d\rfloor,\lceil n/d\rceil\}$. As $n\geq d$, we have $n_{1},...,n_{d}\geq 1$. Let $w_{i}$ be the width of $[k]^{n_{i}}$ and let $l_{i}=k^{n_{i}}/4w_{i}$. By Lemma \ref{estimate}, we have 
	$$w_{i}=\Theta\left(\frac{k^{n_{i}-1}}{\sqrt{n/d}}\right)\mbox{ and } l_{i}=\Theta(k\sqrt{n/d}).$$
	
	By Corollary \ref{chain}, for $i\in [d]$, there exist a partition of $[k]^{n_{i}}$ into chains $C_{i,1},...,C_{i,s_{i}}$ such that $l_{i}\leq |C_{i,j}|\leq 4l_{i}$ for $j\in [s_{i}]$. These chain partitions yield a partition of $[k]^{n}=[k]^{n_{1}}\times...\times [k]^{n_{d}}$ into the collection of cartesian products $G_{j_{1},...,j_{d}}=C_{1,j_{1}}\times...\times C_{d,j_{d}}$, where $(j_{1},...,j_{d})\in [s_{1}]\times...\times [s_{d}]$.
	
Note that the poset $G_{j_{1},...,j_{d}}$ is isomorphic to the $d$-dimensional grid $[|C_{1,j_{1}}|]\times...\times [|C_{d,j_{d}}|]$. Let $m=\max_{i\in [d]}{l_{i}}$, then $m=\Theta(k\sqrt{n/d})$, $|G_{j_{1},...,j_{d}}|=\Theta(1)^{d}m^{d}$, and  $G_{j_{1},...,j_{d}}$ is isomorphic to a subset of the grid $[4m]^{d}$. Hence, as $G_{j_{1},...,j_{d}}\cap S$ does not contain an induced copy of $P$, we have that 
$$|G_{j_{1},...,j_{d}}\cap S|\leq c_{P}(4m)^{d-1}=O(1)^{d}c_{P}\frac{|G_{j_{1},...,j_{d}}|}{m}.$$

But then 
\begin{align*}
|S|&=\sum_{(j_{1},...,j_{d})\in [s_{1}]\times...\times [s_{d}]}|G_{j_{1},...,j_{d}}\cap S|=O(1)^{d}c_{P}\sum_{(j_{1},...,j_{d})\in [s_{1}]\times...\times [s_{d}]}\frac{|G_{j_{1},...,j_{d}}|}{m}\\
   &=O(1)^{d}c_{P}\frac{k^{n}}{m}=O(1)^{d}c_{P}\frac{k^{n-1}}{\sqrt{n/d}}=O(1)^{d}c_{P}\sqrt{d}w
\end{align*}
Hence, setting $C_{P}=C_{0}^{d}c_{P}$ with some large absolute constant $C_{0}$, we have $|S|\leq C_{P}w$.

\end{proof}

\section{Lubell function}\label{sect:lubell}

For a graded poset $Q$ with levels $A_{0},...,A_{n}$, define the \emph{Lubell mass} of the subset $S\subset Q$ as 
$$L_{Q}(S)=\sum_{i=0}^{n}\frac{|S\cap A_{i}|}{|A_{i}|}.$$ 
Note that every level of $Q$ is an antichain, so we have the trivial inequality $$\frac{|S|}{w(Q)}\leq L_{Q}(S).$$
The following strengthening of Theorem \ref{mainthm0} was conjectured by Lu and Milans \cite{LM}, and it was proved by M\'{e}roueh \cite{M}.

\begin{theorem}(M\'{e}roueh \cite{M})\label{Lubell}
	Let $P$ be a poset. There exists a constant $c(P)$ such that for every positive integer $n$, if $S\subset 2^{[n]}$ does not contain a copy of $P$, then $$L_{2^{[n]}}(S)\leq c(P).$$ 
\end{theorem}

Also, a simple consequence of the LYM property is that if $Q$ is a normalized matching poset and $P$ is a chain, then $L_{Q}(S)\leq |P|-1$ holds for every $S\subset Q$ not containing a copy of $P$. This is true because $S$ is the union of at most $|P|-1$ antichains and the Lubell mass of any antichain in $Q$ is at most $1$ by the LYM property. In fact, we have $L_{[k]^{n}}(S)\leq |P|-1$ for every positive integer $k$ and $n$, and $S\subset [k]^{n}$ not containing a copy of the chain $P$.

 Hence, it might be natural to conjecture the following common strengthening of Theorem \ref{mainthm} and Theorem \ref{Lubell}. 
 
 \begin{displayquote}If $P$ is a poset, there exist a constant $C(P)$ such that for every positive integer $k,n$ satisfying $k\geq 2$, $n\geq \dim P$, if $S\subset [k]^{n}$ does not contain a copy of $P$, then $L_{[k]^{n}}(S)\leq C(P)$.
 \end{displayquote}

However, it does not take much effort to show that this conjecture is false, even for small posets $P$. An immediate counterexample is when $P$ is an antichain on $2$ elements and $S\subset [k]^{2}$ is a maximal chain. In this case, $S$ clearly does not contain a copy of $P$, but $L_{[k]^{2}}(S)=1/k+2\sum_{i=1}^{k-1}1/i=\Theta(\log k)$. Let us present another example which shows that this conjecture cannot be saved even by increasing the lower bound on $n$.

Let $K$ be the poset on three elements $a,b,c$ with the only comparable pair $a<b$.

\begin{claim}
	For every positive integer $k,n$ satisfying $n\geq 2$, there exists $S\subset [k]^{n}$ such that $S$ does not contain a copy of $K$ and $$L_{[k]^{n}}(S)\geq \Omega(1)^{n}\log k.$$
\end{claim}

\begin{proof}
	Let $A_{0},...,A_{kn-n}$ be the levels of $[k]^{n}$. Also, let $s=\lfloor\log_{2}k\rfloor-1$, and for $i=0,...,s-1$, let $$B_{i}=\{(a_{1},...,a_{n})\in[k]^{n}:\forall j\in[d],2^{i}\leq a_{j}<2^{i+1}\}.$$
	Also, let $r_{i}=\lfloor(3\cdot 2^{i}-1)n/2\rfloor$ and $$S_{i}=\{(a_{1},...,a_{n})\in B_{i}:a_{1}+...+a_{n}=r_{i}\}.$$ Note that $B_{i}$ is isomorphic to $[2^{i}]^{n}$ and $S_{i}$ is a maximal sized antichain in $B_{i}$. Set $S=\cup_{i=0}^{s}S_{i}$. 
	
	First of all, we show that $S$ does not contain a copy of $K$. Suppose that $0\leq i<j\leq s-1$, then $x<y$ holds for any two elements $x\in B_{i}$ and $y\in B_{j}$. Suppose that $\{a,b,c\}\in S$ is a copy of $K$. As $a$ and $c$ are incomparable, we must have that $a,c\in S_{i}$ for some $i\in \{0,...,s-1\}$. Also, $b$ and $c$ are incomparable, so $b\in S_{i}$ as well. But then $a$ and $b$ are incomparable as $S_{i}$ is an antichain, contradiction.
	
	Now, let us estimate the Lubell mass of $S$. As $S_{i}$ is a maximal sized antichain in $B_{i}$, we have $|S_{i}|=\Theta(2^{i(n-1)}/\sqrt{n})$ by Lemma \ref{estimate}. But $S_{i}$ is contained in the level $A_{r_{i}-n}$, which satisfies
	$$|A_{r_{i}-n}|\leq \binom{r_{i}}{n-1}<\left(\frac{er_{i}}{n-1}\right)^{n-1}\leq O(1)^{n}\cdot 2^{i(n-1)}.$$
	Hence, we have
	$$L_{[k]^{n}}(S)=\sum_{i=0}^{s-1}\frac{|S_{i}|}{|A_{r_{i}-n}|}\geq \Omega(1)^{n}s=\Omega(1)^{n}\log k.$$
\end{proof}

 On the other hand, it is not hard to show that with $n$ fixed, the Lubell mass of a set $S\subset [k]^{n}$ not containing a copy of $P$ is bounded by some multiple of $\log k$.
 
 \begin{claim}
 	Let $P$ be a poset and  let $n\geq \dim P$ be an integer. There exist constants $C(P)$ and $\alpha_{n}$ such that for every positive integer $k$, if $S\subset [k]^{n}$ does not contain a copy of $P$, then $$L_{[k]^{n}}(S)\leq C(P)\alpha_{n}\log k.$$
 \end{claim}

\begin{proof}
	Again, let $A_{0},...,A_{kn-n}$ be the levels of $[k]^{n}$. Let $$Q^{-}=\{(a_{1},...,a_{n})\in [k]^{n}:a_{1}+...+a_{n}\leq (k+1)n/2\},$$ that is $Q^{-}$ is the lower half of $[k]^{n}$, and let $Q^{+}=[k]^{n}\setminus Q^{-}$ be te upper half of $[k]^{n}$. We prove that the Lubell mass of $S^{-}=S\cap Q^{-}$ is bounded by $C_{P}\beta_{n}\log k$, where $C_{P}$ is the constant defined in Theorem \ref{mainthm} and $\beta_{n}$ is some function of $n$. Then, a symmetric argument shows that the Lubell mass of $S^{+}=S\cap Q^{+}$ is also bounded by $C_{P}\beta_{n}\log k$, finishing our proof.

	  Let $s=\lfloor\log_{2}(k+1)n/2\rfloor$, and for $i=0,...,s-1$, let $$C_{i}=\{(a_{1},...,a_{n})\in Q^{-}:2^{i}\leq a_{1}+...+a_{n}<2^{i+1}\}.$$ Also, let $S_{i}=C_{i}\cap S^{-}$. Clearly, $C_{i}\subset [2^{i+1}]^{n}$, so by Theorem \ref{mainthm} and Lemma \ref{estimate}, we have $$|S_{i}|\leq O(C_{P}2^{(i+1)(n-1)}/\sqrt{n}).$$
	  Also, $C_{i}$ is the union of the levels $A_{2^{i}-n},...,A_{r-n}$, where $r=\min\{2^{i+1}-1,(k+1)n/2\}$, which satisfy 
	  $$|A_{2^{i}-n}|\leq...\leq|A_{r-n}|.$$
	  The size of $A_{2^{i}-n}$ is at least $(2^{i}/2n^{2})^{n-1}$, as for any choice $a_{1},...,a_{n-1}\in [2^{i}/n,2^{i}/n+2^{i}/n^{2}]$, there exists a unique $a_{n}$ such that $(a_{1},...,a_{n})\in A_{2^{i}-n}$. 
	  Hence, we have
	  $$L_{[k]^{n}}(S_{i})=\sum_{j=2^{i}-n}^{r-n}\frac{|S_{i}\cap A_{j}|}{|A_{j}|}\leq \frac{(2n^{2})^{n-1}|S_{i}|}{2^{i(n-1)}}\leq O(C_{P}(2n^{2})^{n-1}),$$
	  which gives
	  $$L_{[k]^{n}}(S^{-})=\sum_{i=0}^{s-1}L_{[k]^{n}}(S_{i})=O(C_{P}(2n^{2})^{n-1}s)=O(C_{P}(2n^{2})^{n-1}\log k).$$
\end{proof}

In the light of these results, we believe that the following generalization of Theorem \ref{Lubell} might be true.

\begin{conjecture}
	Let $P$ be a poset. There exists a constant $C(P)$ such that the following holds. For every positive integer $k$ and $n$ satisfying $n\geq \dim P$, if $S\subset [k]^{n}$ does not contain a copy of $P$, then $$L_{[k]^{n}}(S)\leq C(P)\log k.$$
\end{conjecture}

\end{document}